\documentclass{amsart} 

\usepackage{amsthm, amssymb, amsmath, amscd}
\usepackage{eurosym}
\usepackage{amsfonts}
\usepackage{amsmath}
\usepackage{setspace}
\usepackage{graphicx}%

\usepackage[all,cmtip]{xy}

\usepackage{indentfirst}

\newtheorem{theorem}{Theorem}[section]

\newtheorem{cor}[theorem]{Corollary}

\newtheorem{lemma}[theorem]{Lemma}

\newtheorem{prop}[theorem]{Proposition}

\theoremstyle{definition}
\newtheorem{define}[theorem]{Definition}
\newtheorem{ex}[theorem]{Example}
\newtheorem{remark}[theorem]{Remark}

\newcommand{\field}[1]{\mathbb{#1}}

\newcommand{\zkz}{\field{Z}/k\field{Z}}

\begin{document}
\title{Index theory for manifolds with Baas-Sullivan singularities}
\author{Robin J. Deeley}
\date{\today}
\begin{abstract}We study index theory for manifolds with Baas-Sullivan singularities using geometric $K$-homology with coefficients in a unital $C^*$-algebra. In particular, we define a natural analog of the Baum-Connes assembly map for a torsion-free discrete group in the context of these singular spaces. The cases of singularities modelled on $k$-points (i.e., $\zkz$-manifolds) and the circle are discussed in detail. In the case of the former, the associated index theorem is related to the Freed-Melrose index theorem; in the case of latter, the index theorem is related to work of Rosenberg.
\end{abstract}
\maketitle
\section{Introduction}
We consider index theory for manifolds with Baas-Sullivan singularities using the framework of geometric K-homology. We also discuss the Baum-Connes conjecture in this context. Manifolds with Baas-Sullivan singularities were introduced in \cite{Baas} following work in \cite{MS}. They are a generalization of the concept of a $\zkz$-manifold and are well studied objects in differential topology. The reader can find introductions to these objects in \cite{Bot} or \cite[Chapter VIII]{Rud}.
\par
Let $P$ be a smooth, compact, ${\rm spin^c}$-manifold. Informally (see Definition \ref{Pmfld} for the precise definition), a smooth, compact, ${\rm spin^c}$ manifold with Baas-Sullivan singularities modelled on $P$ is the following data: a compact, ${\rm spin^c}$ manifold with boundary, $Q$, whose boundary is diffeomorphic, in a ${\rm spin^c}$-preserving way, to $P \times \beta Q$ for some closed smooth, compact, ${\rm spin^c}$ manifold, $\beta Q$. We denote such an object by $(Q,\beta Q)$ and refered to such an object simply as a $P$-manifold. A $\zkz$-manifold is a $P$-manifold in the case when $P=k$-points. Our goal is the construction of index theorems for such objects; prototypical examples are the Freed-Melrose index theorem for $\zkz$-manifolds \cite{Fre, FM} and the index theorem for $S^1$-manifolds in \cite{Ros}. In addition, the results of this paper are a generalization of our previous work in \cite{Dee1, Dee2}; in those papers, only the case of $\zkz$-manifolds was considered.
\par
In order to construct the index map and Baum-Connes assembly map for these objects, we define an abelian group via Baum-Douglas type cycles; the reader is directed to any of \cite{BD, BD2, BHS, Rav, Wal} for further details on these cycles and the resulting model. Before discussing our construction, we recall the basic definitions of the Baum-Douglas model of K-homology of a finite CW-complex, $X$, with coefficients in a unital $C^*$-algebra, $A$ (see for example \cite{Wal}). A cycle in the theory is a triple, $(M,E,f)$, where $M$ is a smooth, compact, ${\rm spin^c}$-manifold, $E$ is a locally trivial bundle over $M$ with finitely generated, projective Hilbert $A$-modules as fibers (we refer to such an object as an $A$-vector bundle), and $f:M \rightarrow X$ is a continuous function. An abelian group denoted by $K_*(X;A)$ is defined using these cycles modulo an equivalence relation which is geometrically defined. The group, $K_*(X;A)$, is via an explicit map isomorphic to $KK^*(C(X),A)$. When $X=pt$, the isomorphism is given by the map
\begin{equation} \label{isoInCaseOfPt}
(M,E) \in K_*(pt;A) \mapsto {\rm ind_{AS}}(D_{(M,E)}) \in K_*(A) \cong KK^*(\field{C}, A) 
\end{equation}
where ${\rm ind_{AS}}(D_{(M,E)})$ denotes the higher index of the Dirac operator on $M$ twisted by the bundle $E$. Based on this idenification, we can view $K_*(pt;A)$ as a natural ``geometric" home for our indices. 
\par
The basic idea of our construction is to replace the manifold $M$ in the Baum-Douglas model with a $P$-manifold. In more detail, the construction is as follows. As input, we take a finite CW-complex, $X$, a smooth, compact, ${\rm spin^c}$ manifold, $P$, and unital $C^*$-algebra, $A$. A cycle, with respect to this input, is a triple, $((Q, \beta Q), (E_Q, E_{\beta Q}), f)$ where 
\begin{enumerate}
\item $(Q, \beta Q)$ is a smooth, compact, ${\rm spin^c}$ $P$-manifold; 
\item $(E_Q, E_{\beta Q})$ and $f: (Q,\beta Q) \rightarrow X$ are respectively the natural ``$P$-version" of an $A$-vector bundle and continuous function; the precise definitions of these objects are given in Definitions \ref{PCtsMapX}, \ref{PBundle}, and \ref{cycWithBun}.
\end{enumerate}
From these cycles, we construct an abelian group (see Definition \ref{GeoGroup}); it is denoted by $K_*(X;P;A)$. In the case of $P=k$-points and $A=\field{C}$, we obtain the geometric realization of $K$-homology with coefficients in $\zkz$ constructed in \cite{Dee1, Dee2}. Here, we consider the case of more general ``singularities" (i.e., choices of $P$). However, we do not reach full generality since for the proofs of the key properties of $K_*(X;P;A)$, we require $P$ to have a trivial stable normal bundle. This condition implies that there is a ``good" notion of normal bundle for $P$-manifolds.
\par
The main results in regards to fundamental properties of the geometric group are the generalized Bockstein sequence, Theorem \ref{BockTypeSeq}, and a ``uniqueness"  theorem, Theorem \ref{uniquenessThm}. These theorems should be compared respectively with Theorem 1.6 and Proposition 1.14 in \cite[Chapter VIII]{Rud}. Next, in Section 4, we define the assembly map. When the group is torsion-free, this map is the natural analog of the Baum-Connes assembly map. In particular, for a torsion-free group and non-singular manifolds, the equivalence of the definition of assembly considered here and the standard definition is well-known; a detailed proof is given in \cite{land}. It follows from the generalized Bockstein sequence and Five Lemma that the assembly map for $P$-manifolds is an isomorphism whenever the Baum-Connes assembly map is an isomorphism (see Theorem \ref{BCForPMfld}).
\par
As we mentioned above, the indices we construct are elements in $K_*(pt;P;A)$, which is defined via Baum-Douglas type cycles. Based on Equation \ref{isoInCaseOfPt}, one would like to identify this group with a more well-known group -- ideally, via an explicit map which is analogous to the higher index map discussed in Equation \ref{isoInCaseOfPt}. We give the full details of this construction in two examples. These are the case when $P$ is $k$-points and the circle. In these cases, we compute $K_*(pt;P;A)$ and relate the geometrically defined cycles with an analytic index. For $k$-points, the resulting index theorem (see Theorem \ref{higFMThm}) is a higher version of the Freed-Melrose index theorem; for the circle, the resulting index theorem is related to work of Rosenberg, \cite{Ros}. Based on the ``uniqueness" theorem (i.e., Theorem \ref{uniquenessThm}) mentioned above, the examples of $k$-points, the circle, and the $2$-sphere lead to the determination of $K_*(pt;P;A)$ for any choice of $P$ with trivial stable normal bundle and well-defined dimension. The final section of the paper contains a brief discussion of some generalizations.
\par
A summary of the notation introduced to this point is as follows: $X$ is a finite CW-complex, $K^*(X)$ is the K-theory of $X$, $K_*(X)$ is the geometric K-homology of $X$, $A$ is a unital $C^*$-algebra, $K_*(A)$ is the K-theory of $A$, $K_*(X;A)$ is the geometric K-homology of $X$ with coefficients in $A$, $K^*(X;A):=K_*(C(X)\otimes A)$, and $P$ is a smooth, compact, ${\rm spin^c}$-manifold with trivial stable normal bundle. It is convenient to assume that $P$ has a well-defined dimension mod two; that is, we assume the dimensions of the connected components of $P$ are all the same dimension modulo two. An $A$-vector bundle over $X$ is a locally trivial bundle over $X$ with finitely generated, projective Hilbert $A$-modules as fibers. Also, $\Gamma$ is a finitely generated discrete group, $B\Gamma$ is the classifying space of $\Gamma$ (for simplicity, we assume it is a finite CW-complex), $C^*(\Gamma)$ is the reduced group $C^*$-algebra of $\Gamma$. Finally, if $M$ is a smooth, compact, ${\rm spin^c}$ manifold and $E$ is a smooth $A$-vector bundle over it, then $D_{(M,E)}$ denotes the Dirac operator associated to the ${\rm spin^c}$-structure of $M$ twisted by the $A$-vector bundle $E$. This operator has a well-defined higher index, which we denote by ${\rm ind_{AS}}(D_{(M,E)})$ or ${\rm ind}(D_{(M,E)})$; this index is an element of $ K_{{\rm dim}(M)}(A)$.

\section{Preliminaries}
\subsection{Manifolds with Baas-Sullivan singularities}
\begin{define}(see \cite{Baas}) \\
A closed, smooth, ${\rm spin^c}$-manifold with Baas-Sullivan singularities modeled on $P$ (i.e., a closed, smooth, ${\rm spin^c}$ $P$-manifold) is $(Q,\beta Q, \varphi)$ where
\begin{enumerate}
\item $Q$, a smooth, compact, ${\rm spin^c}$manifold with boundary;
\item $\beta Q$ a smooth, compact, ${\rm spin^c}$-manifold;
\item $\varphi: \beta Q \times P \rightarrow \partial Q$ a ${\rm spin^c}$-preserving diffeomorphism. 
\end{enumerate}
We denote such an object as $(Q,\beta Q, \varphi)$ or just as $(Q,\beta Q)$. \label{Pmfld}
\end{define}
There is a similar definition in the case of spin-manifolds and the constructions we give in this paper in the ${\rm spin^c}$-manifold case generalize in a natural way to the case of spin-manifolds and $KO$-theory. The choice to surpress the diffeomorphism, $\varphi$, from the notation causes an abuse of notation: often but not always, we will refer to the projection map $\beta Q \times P \rightarrow \beta Q$ when we should refer to the map $\partial Q \cong \beta Q \times P \rightarrow \beta Q$.
\begin{ex}
Two prototypical examples are $P= k$-points, in which case a $P$-manifold is exactly a $\zkz$-manifold, see any of \cite{Dee1, Fre, MS}) and $P=S^1$, see \cite{Ros}. In $KO$-theory another low-dimensional example is $P=S^1$ with its non-bounding spin-structure, again see \cite{Ros}. 
\end{ex}
\begin{define} \label{PmfldWithBound}
A ${\rm spin^c}$ $P$-manifold with boundary is a pair of  smooth, compact, ${\rm spin^c}$-manifolds with boundary $\bar{Q}$ and, $\beta \bar{Q}$ and a smooth embedding which respects the ${\rm spin^c}$-structures $\bar{\varphi}:\beta \bar{Q} \times P \hookrightarrow \partial \bar{Q}$. The boundary of $(\bar{Q}, \beta \bar{Q}, \bar{\varphi})$ is given by
$$(\partial \bar{Q} - {\rm int}(\varphi(\beta \bar{Q})), \partial \beta \bar{Q}, \bar{\varphi}|_{\partial \beta Q \times P}).$$
Suppressing the embedding from the notation, we have that $(Q, \beta Q)$ is the boundary of $(\bar{Q}, \beta \bar{Q})$ if
$$\partial \beta \bar{Q}  =   \beta Q \hbox{ and }\partial \bar{Q}  =   Q \cup_{\partial Q} \beta \bar{Q} \times P.$$
As in the definition of a closed $P$-manifold, we will usually suppress the embedding (i.e., $\bar{\varphi}$) from the notation.
\end{define}
\begin{define} \label{PCtsMapX}
A continuous map from a $P$-manifold, $(Q, \beta Q, \varphi)$, to a locally compact Hausdorff space (e.g., a finite CW-complex), $X$, is a pair of continuous maps, $f: Q \rightarrow X$ and $f_{\beta Q}: \beta Q \rightarrow X$, such that
$$f|_{\partial Q} = f_{\beta Q} \circ \pi \circ \varphi$$
where $\pi: \beta Q \times P \rightarrow \beta Q$ is the projection map. We will often denote $(f, f_{\beta Q})$ simply as $f$. 
\end{define}
\begin{define}
A continuous map from one $P$-manifold, $(Q, \beta Q, \varphi)$, to another, $(\tilde{Q}, \beta \tilde{Q}, \tilde{\varphi})$, is a pair of continuous maps, $f: Q \rightarrow \tilde{Q}$ and $f_{\beta Q}: \beta Q \rightarrow \beta \tilde{Q}$, such that
$$\tilde{\pi} \circ \tilde{\varphi} \circ f|_{\partial Q} = f_{\beta Q} \circ \pi \circ \varphi $$
where $\pi: \beta Q \times P \rightarrow \beta Q$ and $\tilde{\pi}: \beta \tilde{Q} \times P \rightarrow \beta \tilde{Q}$ are the natural projection maps.
\end{define}
\subsection{Bundles and $K$-theory}
\begin{define} \label{PBundle}
A $P$-vector bundle, $(V_Q,  V_{\beta Q}, \varphi)$ over a $P$-manifold, $(Q, \beta Q, \varphi)$, is a triple, $(V_{Q}, V_{\beta Q}, \alpha)$, where
\begin{enumerate}
\item $V_Q$ is a vector bundle over $Q$;
\item $V_{\beta Q}$ is a vector bundle over $\beta Q$;
\item $\alpha: V_Q|_{\partial Q} \rightarrow \pi^*(V_{\beta Q})$ is a bundle isomorphism which is a lift of $\varphi$; we note that $\pi$ is the projection map $\beta Q \times P \rightarrow \beta Q$.
\end{enumerate}
We will usually denote such a triple simply as $(V_Q, V_{\beta Q})$ and refer to such objects as $P$-bundles.
\end{define}
Many definitons from the theory of vector bundles generalize to $P$-vector bundle: in particular, two $P$-vector bundles, $(V_Q,  V_{\beta Q}, \alpha)$ and $(V^{\prime}_Q,  V^{\prime}_{\beta Q}, \alpha^{\prime})$, over $(Q,\beta Q)$ are isomorphic if there exists vector bundle isomorphisms $\psi_Q : V_Q \rightarrow V^{\prime}_Q$ and $\psi_{\beta Q}: V_{\beta Q} \rightarrow V^{\prime}_{\beta Q}$ such the following diagram commutes:
\begin{center}
$\begin{CD}
(V_Q)|_{\partial Q} @>(\psi_Q)|_{\partial V}>> (V^{\prime}_Q)|_{\partial Q} \\
@V\alpha VV  @VV\alpha^{\prime} V \\
\pi^*(V_{\beta Q}) @>\tilde{\psi}_{\beta Q}>> \pi^*(V^{\prime}_{\beta Q})
\end{CD}$
\end{center}
\vspace{0.2cm}
where $\tilde{\psi}_{\beta Q}$ is the lift of the isomorphism (from $V_{\beta Q}$ to $V^{\prime}_{\beta Q}$) to an isomorhism from $\pi^*(V_{\beta Q})$ to $\pi^*(V^{\prime}_{\beta Q})$. There are similar definitions in the context of $P$-bundles for other standard definitions in vector bundle theory (e.g., pullback bundle, complementary bundle, stably isomorphic, etc).
\begin{prop}
Every $P$-bundle has a complementary $P$-bundle. Moreover, any two complementary $P$-bundles are stably isomorphic. \label{compBundle}
\end{prop}
\begin{proof}
The result is certainly known; we give a proof in the case of complex $P$-bundles. Let $(V_Q, V_{\beta Q})$ be a $P$-bundle and $\pi: \beta Q \times P \rightarrow \beta Q$ be the projection map. Then $V_Q$ has a complementary vector bundle (say $W$); that is, for some $n\in \field{Z}$, we have $V_Q \oplus W \cong Q \times \field{C}^n$.  In particular, $V_Q |_{\partial Q} \oplus W|_{\partial Q} \cong \partial Q \times \field{C}^n$. The definition of $P$-bundle implies that
$$\pi^*(V_{\beta Q}) \oplus W|_{\partial Q} \cong \pi^*(\beta Q \times \field{C}^n).$$
Let $\tilde{W}$ be a vector bundle over $\beta Q$ such that $V_{\beta Q} \oplus \tilde{W} \cong \beta Q \times \field{C}^n$. Then
$$W|_{\partial Q} \oplus \partial Q \times \field{C}^n \cong \pi^*(\tilde{W} \oplus \beta Q \times \field{C}^n).$$
In other words, $V^c=W \oplus Q \times \field{C}^n$ can be given the structure of a $P$-vector bundle; it is also clear that $(V^c,\tilde{W} \oplus \beta Q \times \field{C}^n)$ is a complementary $P$-bundle for $(V_Q, V_{\beta Q})$. 
\par
Finally, let $(W_Q, W_{\beta Q})$ and $(\tilde{W}_Q, \tilde{W}_{\beta Q})$ be two complementary bundle for $(V_Q, V_{\beta Q})$. Then, for some integers $n_1$ and $n_2$, we have
$$(W_Q, W_{\beta Q}) \oplus Q \times \field{C}^{n_1} \cong (W_Q, W_{\beta Q}) \oplus (V_Q, V_{\beta Q}) \oplus (\tilde{W}_Q, \tilde{W}_{\beta Q}) \cong Q\times \field{C}^{n_2}\oplus (\tilde{W}_Q, \tilde{W}_{\beta Q}).$$
\end{proof}
\begin{ex} 
If $(Q, \beta Q)$ is a $P$-manifold, then the tangent bundle of $Q$ as a manifold with boundary is {\it not} a $P$-bundle over $(Q, \beta Q)$. In certain cases, there is a way of working around this issue. In increasing generality, we have the following examples:
\begin{enumerate} 
\item If $(Q,\beta Q)$ is a $\zkz$-manifold, then $(TQ, T(\beta Q \times (0,1])$ can be given the structure of a $\zkz$-bundle.
\item If $P$ is a manifold with trivial tangent bundle, then there exists $n\in \field{N}$ such that $(TQ, T(\beta Q) \oplus \beta Q \times \field{R}^n)$ can be given the structure of a $P$-bundle.
\item If $P$ is a manifold with trivial stable normal bundle, then there exists $n_1$ and $n_2$ such that $(TQ \oplus Q \times \field{R}^{n_1}, T(\beta Q) \oplus \beta Q\times \field{R}^{n_2})$ can be given the structure of a $P$-bundle.
\end{enumerate} 
\end{ex}
\begin{define}
Let $P$ be a manifold with trivial stable normal bundle and $(Q, \beta Q)$ be a $P$-manifold. Then, a $P$-normal bundle for $(Q, \beta Q)$ is a complementary $P$-bundle to $(TQ \oplus Q \times \field{R}^{n_1}, T(\beta Q) \oplus \beta Q\times \field{R}^{n_2})$ where $n_1$ and $n_2$ are as in the previous example. In particular, a normal bundle for a manifold $M$ is a complementary vector bundle of its tangent bundle. 
\end{define}
\begin{remark} \label{norBunCom}
Proposition \ref{compBundle} implies if $P$ has trivial stable normal bundle, then that every $P$-manifold has a $P$-normal bundle and any two $P$-normal bundles are stably isomorphic. 
\end{remark}
\begin{define} 
Let $A$ be a unital $C^*$-algebra. An $A$-vector bundle over $X$ is a locally trivial bundle with finitely generated, projective Hilbert $A$-modules as fibers. An $A$-vector bundle (or more correctly an ``($A$, $P$)-vector bundle") over a $P$-manifold is defined in the same way as in a $P$-vector bundle only one replaces the vector bundles in that definition with $A$-vector bundles. 
\end{define}
\begin{ex} \label{MishBunEx}
Let $\Gamma$ be a discrete group, $B \Gamma$ be its classifying space, $(Q, \beta Q)$ be a $P$-manifold, and $f: (Q, \beta Q) \rightarrow B\Gamma$ be a continuous map. The Mishchenko bundle is given by $\mathcal{L}_{B\Gamma} := E\Gamma \times_{\Gamma} C^*(\Gamma)$. Then, 
$$ V_{Q} = f^*(\mathcal{L}_{B \Gamma}) \hbox{ and } V_{\beta Q} = (f_{\beta Q})^*(\mathcal{L}_{B\Gamma})$$
can be given the structure of a $C^*(\Gamma)$-bundle over $(Q,\beta Q)$. In fact, the pullback along any continuous function $f: (Q,\beta Q) \rightarrow X$ of any $A$-vector bundle over $X$ is an $A$-vector bundle over $(Q, \beta Q)$. 
\end{ex}
\begin{define}
Let $K^0(Q,\beta Q; P;A)$ be the Grothendieck group of the semi-group of isomorphism classes of ($P$,$A$)-vector bundles over $(Q,\beta Q)$.
\end{define}
The group $K^0(Q,\beta Q; P;A)$ shares many properties with standard K-theory. For example, it has similar functorial properties and there is a Thom isomorphism. 
\begin{prop} \label{KthProp}
The following sequence is exact in the middle:
$$K^0(Q,\beta Q; P; A) \rightarrow K^0(Q;A) \oplus K^0(\beta Q;A) \rightarrow K^0(\beta Q\times P; A) $$
where the maps are defined at the level of bundle data to $K$-classes as follows:
\begin{enumerate}
\item $(V_Q,V_{\beta Q}) \in K^0(Q,\beta Q;P;A) \mapsto ([V_Q],[V_{\beta Q}])\in  K^0(Q;A) \oplus K^0(\beta Q;A)$;
\item $ (E_1,E_2)\in  K^0(Q;A) \oplus K^0(\beta Q;A) \mapsto [E_1 |_{\partial Q}] - [\pi^*(E_2)]$.
\end{enumerate}
where $\pi: \beta Q \times P \rightarrow \beta Q$ is the projection map.
\end{prop}
\begin{proof}
The proof is standard; we give the proof for $P$-vector bundles, but the details generalize to the case of $A$-vector bundles. The isomorphism in the definition of a $P$-vector bundle implies that the composition of the two maps is zero. Thus, to show exactness, we need to show that if
$$([E]-[E^{\prime}], [F] - [F^{\prime}]) \in  K^0(Q) \oplus K^0(\beta Q) \mapsto 0 \in K^0(\beta Q \times P),$$
then there exists element in $K^0(Q,\beta Q;P)$ which maps to $([E]-[E^{\prime}], [F] - [F^{\prime}])$. \par
We can assume that $E^{\prime}$ and $F^{\prime}$ are trivial. Using this fact, the assumption that $([E]-[E^{\prime}], [F] - [F^{\prime}])  \mapsto 0$, and basic bundle theory, there exists trivial bundles over $Q$, $\varepsilon_Q$ and $\varepsilon^{\prime}_Q$, and bundle isomorphism
$$ \alpha: (E \oplus \varepsilon_Q)|_{\partial Q} \cong \pi^*(F)\oplus (\varepsilon^{\prime}_Q)|_{\partial Q}. $$
Let $\varepsilon_{\beta Q}$ respectively, $\varepsilon^{\prime}_{\beta Q}$ be the trivial vector bundle over $\beta Q$ of the same rank as $\varepsilon_Q$ respectively, $\varepsilon^{\prime}_Q$. One then checks that the desired class (i.e., the required preimage) is given by
$$[( E \oplus \varepsilon_Q,  F\oplus \varepsilon^{\prime}_{\beta Q}, \alpha)] - [(\varepsilon_Q \oplus \varepsilon^{\prime}_Q, \varepsilon_{\beta Q} \oplus \varepsilon^{\prime}_{\beta Q}, id)].
$$
\end{proof}
Given $\xi \in K^0(Q, \beta Q;P, A)$, we denote by $\xi_{Q}$ or $\xi|_{Q}$ (resp. $\xi_{\beta Q}$ or $\xi|_{\beta Q}$) the image of $\xi$ under the map to $K^0(Q;A)$ (resp. $K^0(\beta Q;A)$); we use similar notation for $K$-theory classes of $P$-manifolds with boundary. A number of proofs require a $K$-theory group that is rather similar to $K^0(Q, \beta Q;P, A)$:
\begin{define}
Let $(Q, \beta Q)$ be $P$-manifold and $W$ be a manifold with boundary such that $\partial W = \beta Q$. Furthermore, let 
$$K^0(Q \cup_{\partial Q} W \times P, W;A)$$ be the Grothendieck group of isomorphism classes of triples of the form
$$(E_{Q\cup_{\partial Q} W\times P}, E_{W}, \varphi) $$
where
\begin{enumerate}
\item $E_{Q\cup_{\partial Q} W\times P}$ is an $A$-vector bundle over $Q\cup_{\partial Q} W\times P$;
\item $E_{W}$ is an $A$-vector bundle over $W$;
\item $\varphi: E_{Q\cup_{\partial Q} W\times P}|_{W\times P} \rightarrow \pi^*(E_W)$ is a bundle isomorphism. We note that $\pi$ is the projection map $W\times P \rightarrow P$.
\end{enumerate}
%Let $(Q, \beta Q)$ be $P$-manifold and $W$ be a manifold with boundary such that $\partial W = \beta Q$. Furthermore, let 
%$$K^0(Q \cup_{\partial Q} W \times P, W;A)$$ be the Grothendieck group of (isomorphism classes of) triples of the form
%$$(E_{Q\cup_{\partial Q} W\times P}, E_{W}, \varphi) $$
%where
%\begin{enumerate}
%\item $E_{Q\cup_{\partial Q} W\times P}$ is an $A$-vector bundle over $Q\cup_{\partial Q} W\times P$;
%\item $E_{W}$ is an $A$-vector bundle over $W$;
%\item $\varphi: E_{Q\cup_{\partial Q} W\times P}|_{W\times P} \rightarrow \pi^*(E_W)$ is a bundle isomormphism. We note %that $\pi$ is the projection map $W\times P \rightarrow P$.
%\end{enumerate}
\end{define}
The proof of the next proposition is similar to the proof of Proposition \ref{KthProp} and is omitted.
\begin{prop}  \label{gluePmfldWithbound}
Using the notation of the previous definition, we have that the following sequence is exact in the middle:
$$K^0(Q \cup_{\partial Q} W \times P, W;A) \rightarrow K^0(Q, \beta Q;A)\oplus K^0(W;A) \rightarrow K^0(\beta Q;A), $$
where the maps are the natural maps on K-theory induced from 
\begin{eqnarray*}
(E_{Q\cup_{\partial Q}W \times P}, E_W) & \mapsto & ( (E_{Q\cup_{\partial Q}W \times P})|_Q, E_W|_{\partial W=\beta Q}), E_W) \\
((F_Q, F_{\beta Q}), V_W) & \mapsto & [V_{W|_{\partial W=\beta Q}}]-[F_{\beta Q}]
\end{eqnarray*}
\end{prop}

\section{Geometric models} \label{geoModSec}
\begin{define} \label{cycWithBun}
A geometric cycle with bundle data, over $X$ with respect to $P$ and $A$ is a triple $((Q,\beta Q), (E_Q, E_{\beta Q}), f)$ where
\begin{enumerate}
\item $(Q,\beta Q)$ is a compact, smooth, ${\rm spin^c}$ $P$-manifold;
\item $(E_Q, E_{\beta Q})$ is a smooth $A$-vector bundle over $(Q, \beta Q)$;
\item $f$ is a continuous map from $(Q,\beta Q)$ to $X$.
\end{enumerate}
\end{define}
\begin{define} \label{cycWithKth}
A geometric cycle with $K$-theory data, over $X$ with respect to $P$ and $A$ is a triple $((Q,\beta Q), \xi , f)$ where
\begin{enumerate}
\item $(Q,\beta Q)$ is a compact, smooth, ${\rm spin^c}$ $P$-manifold;
\item $\xi \in K^0(Q,\beta Q; P; A)$;
\item $f$ is a continuous map from $(Q,\beta Q)$ to $X$.
\end{enumerate}
\end{define}
As in the case of the Baum-Douglas model for $K$-homology, $(Q,\beta Q)$ need not be connected and there is a natural $\field{Z}/2$-grading on cycles defined using the dimensions of the connected components of $(Q,\beta Q)$ modulo two. Furthermore, there is a natural notion of isomorphism for cycles; when we refer to a ``cycle", we mean ``an isomorphism class of a cycle". Addition of cycles is defined using disjoint union; we denote this operation by $\dot{\cup}$.
\begin{define} \label{borForKthGrp}
A bordism or a cycle with boundary with respect to $X$, $P$, and $A$ is $((\bar{Q}, \beta \bar{Q}), \bar{\xi}, \bar{f})$ where
\begin{enumerate}
\item $(\bar{Q},\beta \bar{Q})$ is a compact, smooth, ${\rm spin^c}$ $P$-manifold with boundary;
\item $\bar{\xi} \in K^0(\bar{Q}, \beta \bar{Q}; P;A)$;
\item $\bar{f}: (\bar{Q},\beta \bar{Q}) \rightarrow X$ is a continuous map.
\end{enumerate} 
The boundary of a bordism, $((\bar{Q}, \beta \bar{Q}), \bar{\xi}, \bar{f})$, is given by 
$$((\partial \bar{Q} - {\rm int}(\beta \bar{Q}), \partial \beta \bar{Q}), \bar{\xi}|_{\partial \bar{Q} - {\rm int}(\beta \bar{Q})}, \bar{f}|_{\partial \bar{Q} - {\rm int}(\beta \bar{Q})}).$$
By construction, it is a cycle in the sense of Definition \ref{cycWithKth}.
\end{define}
\begin{define}
Let $((Q,\beta Q), \xi , f)$ be a cycle and $(V_Q,V_{\beta Q})$ a ${\rm spin^c}$ $P$-vector bundle with even dimensional fibers over $(Q,\beta Q)$. We define the vector bundle modification of $((Q,\beta Q), \xi , f)$ by $(V_Q,V_{\beta Q})$ to be 
$$( (S(V_Q\oplus {\bf 1}),S(V_{\beta Q}\oplus {\bf 1})), \pi^*(\xi) \otimes \beta, f\circ \pi) $$
where
\begin{enumerate}
\item ${\bf 1}$ denotes the trivial real line bundle;
\item $\pi$ the vector bundle projection;
\item $\beta$ is the Bott class (see for example \cite[Section 2.5]{Rav}).
\end{enumerate}
We denote the cycle so obtained by $((Q,\beta Q), \xi , f)^{(V_Q,V_{\beta Q})}$; the reader should verify that the resulting triple is a cycle (in the sense of Definiton \ref{cycWithKth}). We will also find it useful to denote the bundle we are modifying by simply as $V_Q$; however, we emphasize that one can only vector bundle modify by $P$-vector bundles, which are ${\rm spin^c}$ and have even dimensional fibers.
\end{define}
\begin{define} \label{GeoGroup}
Let $K_*(X;P;A):=\{(Q,\beta Q), \xi , f)\}/\sim$
where $\sim$ is the equivalence relation generated by bordism and vector bundle modification. We denote the bordism relation by $\sim_{bor}$.
\end{define}
The disjoint union operation gives $K_*(X;P;A)$ the structure of an abelian group. In a similar way, one can define $K_*(X;P;A)$ using cycles with ``bundle data" (see Definition \ref{cycWithBun}). To do so, one needs to use slightly different definitions of bordism and vector bundle modification and add a direct sum/disjoint union relation. The process is completely analogous to the difference between the original Baum-Douglas model \cite{BD} and the model discussed in \cite{Rav}. We will use the ``bundle data" model for the definition of the assembly map and the detailed discussion of the case $P=k$-points and $P=S^1$.
\par
We have the following maps:
\begin{enumerate}
\item $ \Phi_P : K_*(X;A) \rightarrow K_{*+{\rm dim}(P)}(X;A) $
defined at the level of cycles via
$$ \Phi(M,\xi,f)=(M\times P, \pi^*(\xi) , f\circ \pi)$$
where $\pi$ denotes the projection map from $M\times P$ to $M$.
\item $r_P : K_*(X;A) \rightarrow K_*(X;P;A) $
defined at the level of cycles via
$$ r(M, \xi, f) = ((M,\emptyset), \xi, f).$$
\item $\delta_P : K_*(X;P;A) \rightarrow K_{*-{\rm dim}(P)-1} (X;A)$ defined at the level of cycles via
$$\delta((Q,\beta Q), \xi, f) = (\beta Q, \xi|_{\beta Q}, f|_{\beta Q} ).$$
\end{enumerate}
\begin{define} \label{norBordism}
Let $((Q,\beta Q), \xi, f)$ and $((Q^{\prime}, \beta Q^{\prime}), \xi^{\prime}, f^{\prime})$ be two cycles in $K_*(X;P;A)$. Then, these two cycles are normally bordant if there exists normal $P$-bundles $N$ and $N^{\prime}$ such that $((Q,\beta Q), \xi, f)^N \sim_{bor} ((Q^{\prime}, \beta Q^{\prime}), \xi^{\prime}, f^{\prime})^{N^{\prime}}$. If this is the case, then we write 
$$((Q,\beta Q), \xi, f) \sim_{nor} ((Q^{\prime}, \beta Q^{\prime}), \xi^{\prime}, f^{\prime}).$$
We call this relation normal bordism.
\end{define}
The next two lemmas are similar, in both statement and proof, to lemmas in \cite[Sections 4.4 and 4.5]{Rav} (also see \cite[Section 2.2.1]{Dee1}, \cite[Section 4.2]{DeeGeoRelKhom}, and \cite[Section 3.1]{DeeGofSur}). The proofs are omitted.
\begin{lemma} \label{directSumVBM}
Let $((Q,\beta Q), \xi, f)$ be a cycle and $V_1$ and $V_2$ be even rank ${\rm spin^c}$ $P$-bundles over $(Q, \beta Q)$.  Then
$$ ((Q,\beta Q), \xi, f)^{(V_1\oplus V_2)} \sim_{bor}  (((Q,\beta Q), \xi, f)^{V_1})^{(p^*(V_2))} $$
where $p$ is the projection map from $S(V_1\oplus {\bf 1})$ to $Q$.
\end{lemma}
\begin{lemma}
Normal bordism is equal to the equivalence relation generated from bordism and vector bundle modification on cycles with $K$-theory data.
\end{lemma}
\begin{cor}
A cycle $((Q,\beta Q), \xi, f)$ is trivial in $K_*(X;P;A)$ if and only if there exists normal $P$-bundle, $N$, such $((Q,\beta Q), \xi, f)^N$ is a boundary.
\end{cor}
\begin{theorem} \label{BockTypeSeq}
Let $X$ be a finite CW-complex, $A$ be a unital $C^*$-algebra, and $P$ be a smooth compact ${\rm spin^c}$-manifold that has trivial stable normal bundle and well-defined dimension mod two. Then, the following sequence is exact
\begin{center}
$\begin{CD}
K_0(X;A) @>\Phi>> K_{{\rm dim}(P)}(X;A) @>r>> K_{{\rm dim}(P)}(X;P;A) \\
@A\delta AA @. @V\delta VV \\
 K_{{\rm dim}(P)+1}(X;P;A) @<r<<  K_{{\rm dim}(P)+1}(X;A) @<\Phi<<  K_1(X;A)
\end{CD}$
\end{center}
where the maps were defined above, just before Definition \ref{norBordism}.
%\begin{enumerate}
%\item $\Phi$ is defined at the level of cycles via $\Phi(M, \xi, f) = (M\times P, \pi^*(\xi), f\circ \pi)$ where $\pi: M\times P \rightarrow M$ is the projection map;
%\item $r$ is defined at the level of cycles via $ r(M, \xi, f) = ((M,\emptyset), \xi, f)$; 
%\item $\delta$ is defined at the level of cycles via $\delta((Q,\beta Q), \xi, f) = (\beta Q, \xi|_{\beta Q}, f|_{\beta Q})$.
%\end{enumerate}
\end{theorem}
Before beginning the proof, we note that if $P$ is $k$-points and $A=\field{C}$, then this theorem is exactly \cite[Theorem 2.20]{Dee1}. The proof is similar to the proof of that theorem; in particular, the notion of normal bordism plays a key role. 
\begin{proof}
We begin by noting that the proof that the maps are well-defined follows from the compatiblity of the relations used to defined the various groups. In addition, the notion of bordism in the various groups implies that the composition of successive maps is zero.
\par
These observations reduce the proof to showing 
$${\rm ker}(\delta)\subseteq {\rm im}(r), {\rm ker}(r)\subseteq {\rm im}(\Phi), {\rm ker}(\Phi)\subseteq {\rm im}(\delta). $$
Since the argument is very close to the proof of Theorem 2.20 in \cite{Dee1}, we only give a detailed proof that ${\rm ker}(\delta)\subseteq {\rm im}(r)$.
\par
Suppose that $((Q, \beta Q), \xi, f)$ is in ${\rm ker}(\delta)$. We begin by proving that $((Q, \beta Q), \xi, f)$ is equivalent to a cycle in $K_*(X;P;A)$ whose image under $\delta$ is a boundary rather than just trivial. Remark \ref{norBunCom} and \cite[Corollary 4.5.16]{Rav} imply that there exists normal $P$-bundle $(N_Q, N_{\beta Q})$ (for $(Q,\beta Q)$) such that
$(\beta Q, \xi|_{\beta Q}, f|_{\beta Q})^{N_{\beta Q}}$ is a boundary and $((Q, \beta Q), \xi, f)^{(N_Q, N_{\beta Q})}$ is equivalent via the vector bundle modification relation to the original cycle.
\par
Hence, without loss of generality, we can assume there exists bordism with respect to $K_{*-1}(X;A)$, $(W, \eta, g)$, such that $(\beta Q, \xi|_{\beta Q}, f|_{\beta Q}) = \partial (W, \eta, g)$. We form the cycle in $K_*(X;A)$: 
$$(Q \cup_{\partial Q} (W \times P), \xi_{Q} \cup \pi^*(\eta), f \cup g)$$
where $\pi: \beta Q \times P \rightarrow \beta Q$ is the projection map and the $K$-theory class is obtained as follows. Let $\nu$ to be any preimage of $(\xi_Q, \eta)$ in the sequence which, by Proposition \ref{gluePmfldWithbound}, is exact in the middle:
$$K^0(Q \cup_{\partial Q} W \times P, W;A) \rightarrow K^0(Q, \beta Q;A)\oplus K^0(W;A) \rightarrow K^0(\beta Q;A). $$
Then $ \xi_{Q} \cup \pi^*(\eta)$ is defined to be the image of $\nu$ under the natural map $K^0(Q \cup_{\partial Q} W, W;A) \rightarrow K^0(Q\cup_{\partial Q} W;A)$; this class in not unique, but any choice will satisfy the properties required in the rest of the proof.
\par
The proof will be complete upon showing that $r(Q \cup_{\partial Q} (W \times P), \xi_{Q} \cup \pi^*(\eta), f \cup g)$ is equivalent to $((Q, \beta Q), \xi, f)$. This follows from the following bordism with respect to the group $K_*(X;P;A)$:
$$((Q \cup_{\partial Q} (W \times P) \times [0,1], W\times \{0\}), \bar{\xi}, (f \cup g)\circ \tilde{\pi})$$
where $\tilde{\pi}$ is the projection map $(Q \cup_{\partial Q} (W \times P)) \times [0,1] \rightarrow (Q \cup_{\partial Q} (W \times P))$ and $\bar{\xi}$ is a $K$-theory class such that
$$ \bar{\xi}|_{Q \cup_{\partial Q} (W \times P)\times \{1\}} = \xi_{Q} \cup \pi^*(\eta) \hbox{ and }  \bar{\xi}|_{Q \cup_{\partial Q} (W \times P)\times \{0\}}= \nu.  $$
Such a class can be constructed by more or less pulling back the class $\nu$ discussed above.
\end{proof}
\begin{theorem} \label{uniquenessThm}
Let $X$ be a finite CW-complex, $A$ be a unital $C^*$-algebra, and $P$ and $P^{\prime}$ be smooth, compact, ${\rm spin^c}$-manifolds of dimensions $n$ and $n'$ respectively. Assume also that both $P$ and $P^{\prime}$ have trivial stable normal bundle. If $(P, [P\times \field{C}])\sim (P^{\prime}, [P^{\prime}\times \field{C}])$ as elements in $K_*(pt)$, then $K_*(X;P;A) \cong K_*(X;P^{\prime};A)$. Moreover, the isomorphism is natural with respect to both $X$ and $A$.
\end{theorem}
\begin{proof}
A number of projection maps are required in the proof; we use the following notation: let $M_1$ and $M_2$ be two manifolds possibly with boundary. Then, for $i=1$ and $2$, we let $\pi^{M_1\times M_2}_{M_i} : M_1 \times M_2 \rightarrow M_i$ denote the projection map.
\par
The equivalence relation on cycles in $K_*(pt)$ is equivalent to normal bordism \cite[Corollary 4.5.16]{Rav}. Hence, there exists normal bundles, $N$ and $N^{\prime}$, over respectively $P$ and $P^{\prime}$, such that
\begin{equation}
\label{norBorInUniProof}
(P, [P\times \field{C}])^{N} \sim_{bor} (P^{\prime}, [P^{\prime}\times \field{C}])^{N^{\prime}}.
\end{equation}
Moreover, since both $P$ and $P^{\prime}$ have trivial stable normal bundle and normal bundles are stably isomorphic, we can and will assume that both $N$ and $N^{\prime}$ are trivial bundles. However, since $P$ and $P^{\prime}$ are not necessarily connected, the ranks of these trivial bundles may vary over the connected components of $P$ and $P^{\prime}$. 

Our first goal is to show that we can take $N$ and $N^{\prime}$ each with constant rank. Decompose $P$ and $P^{\prime}$ as follows:
$$ P= P_1 \dot{\cup} \ldots \dot{\cup} P_k \hbox{ and }P^{\prime} = P^{\prime}_1 \dot{\cup} \ldots \dot{\cup} P^{\prime}_{k^{\prime}} $$
where 
\begin{enumerate}
\item ${\rm rank}(N|_{P_i})$ is constant for each $1 \le i \le k$;
\item ${\rm rank}(N^{\prime}|_{P^{\prime}_j})$ is constant for each $1 \le j \le k^{\prime}$;
\item ${\rm rank}(N_1) > {\rm rank}(N_2) > \ldots >{\rm rank}(N_k)$ where $N_i:=N|_{P_i}$;
\item ${\rm rank}(N^{\prime}_1) > {\rm rank}(N^{\prime}_2) > \ldots >{\rm rank}(N^{\prime}_{k^{\prime}})$ where $N^{\prime}_j:=N^{\prime}|_{P^{\prime}_j}$.
\end{enumerate}
The existence of the bordism in Equation \eqref{norBorInUniProof} and dimensional reasons imply that $k=k^{\prime}$ and that, for each $i=1, \ldots k$, 
$$ {\rm rank}(N_1) - {\rm rank}(N_i) = {\rm rank}(N^{\prime}_1) - {\rm rank}(N_i).$$
Furthermore, if $W$ is the manifold with boundary in a bordism from Equation \eqref{norBorInUniProof}, then 
$$W= W_1 \dot{\cup} \ldots \dot{\cup} W_k $$
where, for each $i=1, \ldots k$, $W_i$ is a bordism between $P_i^{N_i}$ and $(P^{\prime}_i)^{N^{\prime}_i}$.
\par
Let $V$ be the vector bundle over $P$ which is the trivial bundle of rank equal to ${\rm rank}(N_1)-{\rm rank}(N_i)$ on each $P_i$ where $i=1, \ldots, k$. Likewise, let $V^{\prime}$ be the vector bundle over $P^{\prime}$ which is the trivial bundle of rank equal to
$${\rm rank}(N_1)-{\rm rank}(N_i)={\rm rank}(N^{\prime}_1)-{\rm rank}(N^{\prime}_i)$$ on each $P^{\prime}_j$ where $j=1, \ldots, k$. Using the fact that trivial bundles of constant rank extend across the bordisms $W_1, \ldots W_k$ and \cite[Lemma 4.4.3]{Rav}, we obtain
\begin{eqnarray*}
 (P, [P\times \field{C}])^{N \oplus V} & \sim_{bor} & ((P, [P\times \field{C}])^N)^{\pi(V)} \\
&  \sim_{bor} & ((P^{\prime}, [P^{\prime}\times \field{C})^N)^{(\pi^{\prime})^*(V^{\prime})} \\
& \sim_{bor} & (P^{\prime}, [P^{\prime}\times \field{C}])^{N^{\prime} \oplus V^{\prime}} 
\end{eqnarray*}
where $\pi$ and $\pi^{\prime}$ are the projection maps in the vector bundle modifications of $P$ by $N$ and $P^{\prime}$ by $N^{\prime}$ respectively. Finally, by construction, the ranks of $N \oplus V$ and $N^{\prime} \oplus V^{\prime}$ are each constant. Thus, with loss of generality, we can assume that 
$$(P, [P\times \field{C}])^{N} \sim_{bor} (P^{\prime}, [P^{\prime}\times \field{C}])^{N^{\prime}}$$
where $N$ and $N^{\prime}$ are trivial bundles with constant rank.
\par
Let $2n$ (resp. $2n^{\prime}$) be the rank of $N$ (resp. $N^{\prime}$) and $\beta_{2n}$ (resp. $\beta_{2n^{\prime}}$) be the Bott class on $S^{2n}$ (resp. $S^{2n^{\prime}})$; the construction of the Bott class can be found in \cite[Section 2.5]{Rav}. Using this notation and the definition of vector bundle modification, we have that
$$(P \times S^{2n}, (\pi^{P\times S^{2n}}_{S^{2n}})^*(\beta_{2n})) \sim_{bor} (P^{\prime}\times S^{2n^{\prime}}, (\pi^{P^{\prime}\times S^{2n^{\prime}}}_{S^{2n^{\prime}}})^*(\beta_{2n^{\prime}}))$$
where, following the notation introduced at the start of the proof,
$$\pi^{P\times S^{2n}}_{S^{2n}}: P \times S^{2n} \rightarrow S^{2n} \hbox{ and } \pi^{P^{\prime}\times S^{2n^{\prime}}}_{S^{2n^{\prime}}}: P^{\prime}\times S^{2n^{\prime}} \rightarrow S^{2n^{\prime}}$$ are the projection maps. Let $(W, \nu)$ denote a fixed choice of bordism between these two cycles; we note that $W$ is a smooth, compact ${\rm spin^c}$-manifold with boundary and $\nu \in K^0(W)$. Let $\partial_0 W$ (resp. $\partial_1 W$) denote the component of $\partial W$ diffeomorphic to $P\times S^{2n}$ (resp. $P^{\prime}\times S^{2n^{\prime}}$). Moreover, we fix vector bundles $F$ and $\tilde{F}$ over $W$, $F_{S^{2n}}$ and $\tilde{F}_{S^{2n}}$ over $S^{2n}$, and $F^{\prime}_{S^{2n^{\prime}}}$ and $\tilde{F}^{\prime}_{S^{2n^{\prime}}}$ over $S^{2n^{\prime}}$ such that 
\begin{eqnarray}
\nu & = & [F]-[\tilde{F}],  \label{FixBunIsoFirst} \\
F|_{\partial_0 W}  \cong  (\pi^{P\times S^{2n}}_{S^{2n}})^*(F_{S^{2n}}) & \hbox{ and } & \tilde{F}|_{\partial_0 W}  \cong   (\pi^{P\times S^{2n}}_{S^{2n}})^*(\tilde{F}_{S^{2n}}), \label{FixBunIsoSecond} \\
F^{\prime}|_{\partial_1 W}  \cong  (\pi^{P^{\prime}\times S^{2n^{\prime}}}_{S^{2n^{\prime}}})^*(F^{\prime}_{S^{2n^{\prime}}}) & \hbox{ and } &\tilde{F}^{\prime}|_{\partial_1 W}  \cong (\pi^{P^{\prime}\times S^{2n^{\prime}}}_{S^{2n^{\prime}}})^*(\tilde{F}^{\prime}_{S^{2n^{\prime}}}).  \label{FixBunIsoLast}
\end{eqnarray}
\par
We use this data including the explicit choices of vector bundle isomorphisms in the previous equations to define the isomorphism; its definition is somewhat involved. 
\par
Let $((Q, \beta Q), \xi ,f)$ be a cycle in $K_*(X;P;A)$ and $(E_Q, E_{\beta Q})$ and $(\tilde{E}_Q, \tilde{E}_{\beta Q})$ be $A$-vector bundles over $(Q, \beta Q)$ such that $\xi=[(E_Q, E_{\beta Q})]-[(\tilde{E}_Q, \tilde{E}_{\beta Q})]$. 
\par
Using the decomposition 
$$\partial (\beta Q \times W)= (\beta Q \times \partial_0 W) \dot{\cup} (\beta Q \times \partial_1 W) \cong (\beta Q \times S^{2n}\times P) \dot{\cup} (\beta Q \times S^{2n^{\prime}} \times P^{\prime}).$$
form the $P^{\prime}$-manifold 
$$ (Q\times S^{2n} \cup_{\partial Q \times S^{2n}} \beta Q \times W, \beta Q \times S^{2n^{\prime}})$$
\par
The class in $K^0((Q\times S^{2n} \cup_{\partial Q \times S^{2n}} \beta Q \times W, \beta Q \times S^{2n^{\prime}}))$ is constructed as follows. For a cocycle $(E_Q, E_{\beta Q})$, we let
$$ \gamma(E_Q, E_{\beta Q}):=[(E^{\prime}, E^{\prime}_{\beta Q})] - [(\hat{E}^{\prime}, \hat{E}^{\prime}_{\beta Q})]$$
where
\begin{eqnarray*}
E^{\prime} & = & (\pi^{Q\times S^{2n}}_Q)^*(E_Q) \otimes (\pi^{Q\times S^{2n}}_{S^{2n}})^*(F_{S^n}) \cup_{P\times \beta Q \times S^{2n}} (\pi^{\beta Q \times W}_{\beta Q})^*(E_{\beta Q})\otimes (\pi^{\beta Q \times W}_W)^*(F) \\ 
E^{\prime}_{\beta Q} & = & (\pi^{\beta Q \times S^{2n^{\prime}}}_{\beta Q})^*(E_{\beta Q}) \otimes (\pi^{\beta Q \times S^{2n^{\prime}}}_{S^{2n^{\prime}}})^*(F^{\prime}_{S^{2n^{\prime}}}) \\
\hat{E}^{\prime} & = &  (\pi^{Q\times S^{2n}}_Q)^*(E_Q) \otimes (\pi^{Q\times S^{2n}}_{S^{2n}})^*(\tilde{F}_{S^n}) \cup_{P\times \beta Q \times S^{2n}} (\pi^{\beta Q \times W}_{\beta Q})^*(E_{\beta Q})\otimes (\pi^{\beta Q \times W}_W)^*(\tilde{F})  \\
\hat{E}^{\prime}_{\beta Q} & = &  (\pi^{\beta Q \times S^{2n^{\prime}}}_{\beta Q})^*(E_{\beta Q}) \otimes (\pi^{\beta Q \times S^{2n^{\prime}}}_{S^{2n^{\prime}}})^*(\tilde{F}^{\prime}_{S^{2n^{\prime}}})
\end{eqnarray*}
We make note of the use of the identification $\partial Q \times S^{2n} \cong P\times \beta Q \times S^{2n}\cong \partial_0 W \times \beta Q$ and that the required clutching isomorphisms are also included in the data (via the definition of $P$-bundle and the fixed vector bundle isomorphisms in Equations \ref{FixBunIsoSecond} and \ref{FixBunIsoLast}). Then, for $\xi=[(E_Q, E_{\beta Q})]-[(\tilde{E}_Q, \tilde{E}_{\beta Q})]$, we let
\begin{equation}
\xi^{\prime}:=\gamma(E_Q, E_{\beta Q})- \gamma(\tilde{E}_Q, \tilde{E}_{\beta Q})
\end{equation}
Using standard methods in topological $K$-theory, one checks that $\xi^{\prime} \in K^0 (Q\times S^{2n} \cup_{\partial Q \times S^{2n}} \beta Q \times W, \beta Q \times S^{2n^{\prime}})$ and $\xi^{\prime}$ do not depend on the choice of cocycles $(E_Q, E_{\beta Q})$ and $(\tilde{E}_Q, \tilde{E}_{\beta Q})$, but only on the class $\xi$ and, of course, on the choice of $W$, $F$, $\tilde{F}$, etc.
\par
Finally, the function, denoted by $g$, is defined to be $f$ on $Q$ and $f\circ \pi_W$ on $\beta Q \times W$.
\par
Let $\Psi_{(W,\nu)} : K_*(X;P;A) \rightarrow K_*(X;P^{\prime};A)$ be the map defined via the process just described; that is, 
$$\Psi_{(W,\nu)}((Q, \beta Q), \xi, f):=((Q\times S^{2n} \cup_{\partial Q \times S^{2n}} \beta Q \times W, \beta Q \times S^{2n^{\prime}}), \xi^{\prime}, g). $$
We must show that this map is well-defined. 
\par
Suppose that $((Q,\beta Q), \xi, f)$ is the boundary in the sense of Definition \ref{borForKthGrp} of the bordism $((\bar{Q}, \beta \bar{Q}), \bar{\xi}, \bar{f})$. Definitions \ref{PmfldWithBound} and \ref{borForKthGrp} imply that $\partial \beta \bar{Q}= \beta Q$; hence, we can form the closed, smooth, ${\rm spin^c}$ manifold, 
$$Q\times S^{2n} \cup_{\partial Q \times S^{2n}} \beta Q \times W \cup_{\beta Q \times S^{2n^{\prime}} \times P^{\prime}} \beta \bar{Q} \times S^{2n^{\prime}}\times P^{\prime}$$
There are also an associated K-theory class and function: For the K-theory class, we take 
$$ \xi^{\prime}|_{(Q\times S^{2n} \cup_{\partial Q \times S^{2n}} \beta Q \times W} \cup (\pi^{\beta \bar{Q} \times S^{2n^{\prime}}}_{\beta \bar{Q}})^*(\bar{\xi})|_{\beta \bar{Q}\times S^{2n^{\prime}}}$$
and, for the function, we take
$$ g \cup (\bar{f}|_{\beta \bar{Q}} \circ \pi^{\beta \bar{Q} \times S^{2n^{\prime}}}_{\beta \bar{Q}}).$$
Our goal is to show that this manifold is a boundary; we must also check that the bordism respects the K-theory class and continuous function. Definition \ref{PmfldWithBound} implies that $(Q \cup \beta \bar{Q}\times P)\times S^{2n}$ is a boundary. Moreover, the K-theory and function also respect this construction. Hence, it is a bordism in the sense of the Baum-Douglas model of K-homology. 
\par
Furthermore, we will show that
$$Q\times S^{2n} \cup_{\partial Q \times S^{2n}} \beta Q \times W \cup_{\beta Q \times S^{2n^{\prime}} \times P^{\prime}} \beta \bar{Q} \times S^{2n^{\prime}}\times P^{\prime} \dot{\cup} -((Q \cup \beta \bar{Q}\times P)\times S^{2n})$$ 
is bordant to
$$ (\beta \bar{Q}\times S^{2n}\times P) \cup_{\beta Q \times S^{2n} \times P}   (\beta Q \times W) \cup_{\beta Q \times S^{2n^{\prime}} \times P^{\prime}} (\beta \bar{Q} \times S^{2n^{\prime}}\times P^{\prime})$$
The bordism we have in mind is formed in a similar manner to the construction of the ``pair of pants" bordism by gluing the manifolds $[0,1]\times [0,1]$ and $S^1\times [0,1]$ together; the details are as follows: the manifold with boundary in the bordims is constructed by applying the ``straightening the angle" technique (see \cite{CFPerMap} or \cite[Appendix]{Rav}) to the manifold with corners formed by gluing $ Q \times S^{2n} \times [0,1] $ to 
$$ \left( (\beta \bar{Q}\times S^{2n}\times P) \cup_{\beta Q \times S^{2n} \times P}   (\beta Q \times W) \cup_{\beta Q \times S^{2n^{\prime}} \times P^{\prime}} (\beta \bar{Q} \times S^{2n^{\prime}}\times P^{\prime}) \right) \times [0,1]$$
along $(\partial Q \times S^{2n}) \times [0,1] \cong \left ( (\beta Q \times \partial_0 W) \times [0,1] \right) \times \{1\}$. We note that we consider $\left( (\beta Q \times \partial_0 W) \times [0,1] \right)$ as a subspace of 
$$ \left( (\beta \bar{Q}\times S^{2n}\times P) \cup_{\beta Q \times S^{2n} \times P}   (\beta Q \times W) \cup_{\beta Q \times S^{2n^{\prime}} \times P^{\prime}} (\beta \bar{Q} \times S^{2n^{\prime}}\times P^{\prime}) \right) $$
Finally, by ``straightening the angle" (see \cite{CFPerMap} or \cite[Appendix]{Rav}) of $\beta \bar{Q} \times W$, we obtain a smooth compact ${\rm spin}^c$ manifold with boundary; it has boundary 
$$\partial \beta \bar{Q} \times W \cup \beta Q \times \partial W=(\beta \bar{Q}\times S^{2n}\times P) \cup_{\beta Q \times S^{2n} \times P}   (\beta Q \times W) \cup_{\beta Q \times S^{2n^{\prime}} \times P^{\prime}} (\beta \bar{Q} \times S^{2n^{\prime}}\times P^{\prime})$$
Combining these three observations completes the construction of the required bordism at least at the level of the ``manifold part" of the cycle. That this bordism respects the continuous function data also follows from standard results in bordism theory. In particular, by \cite{CFPerMap}, the ``straightening the angle" process respects the continuous functions involved in the bordism. That the bordism respects the K-theory data is a bit move involved, but it also follows from the fact that the ``straightening the angle" process respects K-theory data (see \cite[Appendix]{Rav}). This completes the proof that $\Psi_{(W,\nu)}$ respects the bordism relation.
\par
For vector bundle modification, let $((Q, \beta Q), \xi, f)$ be a cycle and $((Q^{\prime}, \beta Q^{\prime}), \xi^{\prime}, f^{\prime})$ be its image under $\Psi_{(W,\nu)}$. Also, let $(V_Q, V_{\beta Q})$ be a ${\rm spin^c}$ $P$-bundle of even rank over $((Q, \beta Q)$. Form the vector bundle $ (\pi^{Q\times S^{2n}}_Q)^*(V_Q) \cup_{\partial Q} (\pi^{W\times \beta Q}_{\beta Q})^*(V_{\beta Q})$ over $Q^{\prime}$; it is also ${\rm spin^c}$, of even rank, and can be given the structure of a $P^{\prime}$-bundle. As such, we can consider the vector bundle modification of $((Q^{\prime}, \beta Q^{\prime}), \xi^{\prime}, f^{\prime})$ by this bundle. Moreover, the definition of vector bundle modification implies that 
 $$\left( (Q^{\prime}, \beta Q^{\prime}), \xi^{\prime}, f^{\prime} \right)^{ \left( (\pi^{Q\times S^{2n}}_Q )^*(V_Q) \cup_{\partial Q} (\pi^{W\times \beta Q}_{\beta Q})^*(V_{\beta Q}) \right) } = \Psi_{(W,\nu)} \left( ((Q, \beta Q), \xi, f \right)^{(V_Q, V_{\beta Q})}$$
This completes the proof that $\Psi_{(W,\nu)}$ is well-defined.
\par
To see that this map is an isomorphism, we can use the generalized Bockstein sequence and Five Lemma. However, there is a more direct approach as follows. We recall that $(-W, \nu)$ denotes the opposite of $(W, \nu)$ and that $(-W,\nu)$ is therefore a bordism from $(P^{\prime}, [P^{\prime}\times \field{C}])^{N^{\prime}}$ to $(P,[P\times \field{C}])^N$. Using the same choice for vector bundles representing $\nu$ (see Equations \ref{FixBunIsoFirst} to \ref{FixBunIsoLast}), we have the map 
$$\Psi_{(-W,\nu)}: K_*(X;P^{\prime};A) \rightarrow K_*(X;P;A).$$
Moreover, it follows from standard results in bordism theory (essentially the observation that the double of a manifold is a boundary) imply that $\Psi_{(-W,\nu)}$ is the inverse of $\Psi_{(W,\nu)}$. 
\par
Finally, that isomorphism constructed in this proof is natural with respect to $X$ and $A$ follows from the explicit nature of the map.
\end{proof}
\begin{remark} \label{remThmDepKhomClass}
The previous theorem implies that if we can compute $K_*(X;P;A)$ for $P$ equaling $k$-points, $S^1$, and $S^2$, then we have determined $K_*(X;P;A)$ for any $P$, which satisfy the conditions in the statement of the previous theorem and such that $(P, [P\times \field{C}])$ is equivalent in $K_*(pt)$ to $(k$-points,$[k$-points$\times \field{C}])$, $(S^1, [S^1\times \field{C}])$, or $(S^2, [S^2\times \field{C}])$. We note that this list of classes contains all elements in $K_*(pt)$ and that these three examples will be discussed in detail in Section \ref{exSec}. 
\end{remark}
\section{The assembly map for $P$-manifolds}
Let $\Gamma$ be a finitely generated discrete group, $B\Gamma$ be the classifying space of $\Gamma$, $C^*(\Gamma)$ be the reduced group $C^*$-algebra of $\Gamma$, although similar results hold for the full group $C^*$-algebra. For simpility, we assume $B\Gamma$ is a finite CW-complex. Recall (see \cite{Wal}) that the Baum-Connes assembly map, $\mu: K_*(B\Gamma) \rightarrow K_*(pt;C^*(\Gamma))$, can be defined at the level of geometric cycles as follows: 
$$(M,E,f) \mapsto (M,E\otimes_{\field{C}} f^*(\mathcal{L}_{B\Gamma}))$$
where $\mathcal{L}_{B\Gamma}$ is the Mishchenko line bundle; it was defined in Example \ref{MishBunEx}. This definition of assembly generalizes to the $P$-manifold setting as follows.
\begin{define}
The assembly map, $\mu_{P}:K_*(B\Gamma;P) \rightarrow K_*(pt;P;C^*(\Gamma))$, is defined at the level of cycles via
$$((Q,\beta Q), (E_Q,E_{\beta Q}), f) \mapsto ((Q,\beta Q), (E_Q \otimes_{\field{C}} f^*(\mathcal{L}_{B\Gamma}), E_{\beta Q} \otimes_{\field{C}} (f_{\beta Q})^*(\mathcal{L}_{B\Gamma}))).$$
\end{define}
\begin{theorem} \label{BCForPMfld}
The map $\mu_P$ is well-defined. Moreover, the following diagram is commutative:
{\footnotesize \minCDarrowwidth10pt\begin{center}
$\begin{CD}
@>>> K_0(B\Gamma) @>\Phi_P>> K_{{\rm dim}(P)}(B\Gamma) @>r_P>> K_0(B\Gamma;P) @>\delta_P>> K_{{\rm dim}(P)-1}(B\Gamma) @>>> \\
@. @V\mu VV @V\mu VV @V\mu_P VV @V\mu VV  @. \\
@>>> K_0(pt;C^*(\Gamma)) @>\Phi_P >> K_{{\rm dim}(P)}(pt; C^*(\Gamma)) @>r_P >> K_0(pt;P;C^*(\Gamma)) @>\delta_P >> K_{{\rm dim}(P)-1}(pt;C^*(\Gamma)) @>>> 
\end{CD}$
\end{center}}
\noindent
In particular, if $\mu: K_*(B\Gamma) \rightarrow K_*(pt;C^*(\Gamma))$ is an isomorphism, then $\mu_P$ is an isomorphism.
\end{theorem}
\begin{proof}
The first two statements follow by observing that all the groups are defined using geometric cycles and the maps are defined at the level of these cycle. As such, the proof is a matter of showing the relations on the various cycles are compatible, which can be checked directly. The last statement in the theorem follows from the first two and the Five Lemma.
\end{proof}

\section{Examples} \label{exSec}
In this section, we work with the bundle model of $K_*(X;P;A)$. That is, we use cycles as in Definition \ref{cycWithBun}. The definitions of the operations (e.g., addition, opposite, etc) and the relation (e.g., bordism, vector bundle modification, etc) are the natural ``$P$-manifold version" of those given in any of \cite{BD, BHS, Wal}. We include the definition of bordism in this context as a prototypical example. 
\begin{define}
\label{borForBunGrp}
A bordism or a cycle with boundary with respect to $X$, $P$, and $A$ in the bundle model is $((\bar{Q}, \beta \bar{Q}), (\bar{E}_Q, \bar{E}_{\beta Q}), \bar{f})$ where
\begin{enumerate}
\item $(\bar{Q},\beta \bar{Q})$ is a compact, smooth, ${\rm spin^c}$ $P$-manifold with boundary;
\item $(\bar{E}_Q, \bar{E}_{\beta Q})$ is a smooth $A$-vector bundle over $(\bar{Q}, \beta \bar{Q})$;
\item $\bar{f}: (\bar{Q},\beta \bar{Q}) \rightarrow X$ is a continuous map.
\end{enumerate} 
The boundary of a bordism, $((\bar{Q}, \beta \bar{Q}), (\bar{E}_Q, \bar{E}_{\beta Q}), \bar{f})$, is given by 
$$((\partial \bar{Q} - {\rm int}(\beta \bar{Q}), \partial \beta \bar{Q}), ((\bar{E}_Q) |_{\partial \bar{Q} - {\rm int}(\beta \bar{Q})}, \bar{E}_{\beta Q}|_{\partial \beta \bar{Q}}), \bar{f}|_{\partial \bar{Q} - {\rm int}(\beta \bar{Q})}).$$
By construction, it is a cycle in the sense of Definition \ref{cycWithBun}.
\end{define}
In the first example below we make use of K-homology with coefficients in $\zkz$. The reader can find details of the analytic (i.e., KK-theory) construction of this group in \cite{Schochet}. In particular, results in \cite{Schochet} imply that we can define the K-homology of $X$ with cofficients in $\zkz$ as follows:
$$K_*(X;\zkz):= KK^*(C(X),C_{\phi}).$$
where $C_{\phi}$ is the mapping cone of the unital inclusion of the complex numbers in the $k$ by $k$ matrices. In similar way,
$$K_*(X;A;\zkz):=KK^*(C(X), A\otimes C_{\phi}) \hbox{ and } K_*(A;\zkz):=K_*(A\otimes C_{\phi})$$
\begin{ex}
In this example, we work with the bundle model of $K_*(X;P;A)$ in the case when $P=k$-points. Results in \cite{Dee1, Dee2} imply that
$K_*(X;P;\field{C}) \cong K_*(X;\field{C};\zkz)$. The methods used in those papers can be generalized to the case of general $A$; that is, 
$$K_*(X;P;A) \cong K_*(X;A;\zkz)$$
Moreover, in the case when $X=pt$ and $A=\field{C}$, the map 
$$K_0(pt;\{ k\hbox{-points}\};\field{C}) \cong K_0(pt;\zkz) \cong \zkz$$
can be taken to be the map defined at the level of cycle 
$$((Q, \beta Q), (E_{Q}, E_{\beta Q})) \mapsto {\rm ind_{APS}}(D_{Q, E_Q}) \hbox{ mod }k.$$
where ${\rm ind}_{APS}(\: \cdot \:)$ denotes the Atiyah-Patodi-Singer index (see \cite{APS1}). 
\par
In the case of general $A$, using results in \cite{Schochet}, one has the six-term exact sequence:
\begin{center}
$\begin{CD}
K_0(X;A) @>k>> K_0(X;A) @>r>> K_0(X;A;\zkz) \\
@A\delta AA @. @V\delta VV \\
 K_1(X;A;\zkz) @<r<<  K_1(X;A) @<k<<  K_1(X;A)
\end{CD}$
\end{center}
If $K_*(A)$ contains no torsion of order $k$, then this six-term exact sequence reduces to the following two short exact sequences:
$$0  \rightarrow K_0(pt;A) \rightarrow K_0(pt;A) \rightarrow K_0(pt;A;\zkz) \rightarrow 0 $$
$$0  \rightarrow K_1(pt;A) \rightarrow K_1(pt;A) \rightarrow  K_1(pt;A;\zkz) \rightarrow 0.$$
In this case, an index map 
$$K_*(pt;\{k\hbox{-points}\};A) \rightarrow K_*(pt;A;\zkz) \cong K_*(A;\zkz)$$ 
can be defined using higher APS-index theory; the reader can find more details on this theory in \cite{LP} and the references therein. \par
A key object in higher APS-index theory is the notion of a spectral section (see \cite{LP}). There exists a spectral section with respect to a ${\rm spin^c}$-manifold $M$ and an $A$-bundle over it if and only if the index of the assoicated twisted Dirac operator vanishes. Hence, the existence of a spectral section for the boundary of a manifold with boundary $W$ and the restriction of an $A$-vector bundle over $W$ to $\partial W$ follows from the cobordism invariance of the higher index. In the case of a $\zkz$-manifold $(Q, \beta Q)$ and $A$-vector bundle $(E_{Q}, E_{\beta Q})$, we need to choose a spectral section with respect to $\beta Q$ and $E_{\beta Q}$; such a spectral section does not in general exist. In general, we have
$$0={\rm ind}(D_{\partial Q, E_{Q}|_{\partial Q}})= k \cdot {\rm ind}(D_{\beta Q, E_{\beta Q}}).$$ 
However, if we assume that $K_*(A)$ contains no torsion of order $k$, then this equality implies that ${\rm ind}(D_{\beta Q, E_{\beta Q}})=0$ and hence that a spectral section for $\beta Q$ and $E_{\beta Q}$ exists (see \cite[Section 2]{LP}).
\par
As such, under the assumption that $K_*(A)$ contains no torison of order $k$, we define the index map at the level of cycles via
$$((Q, \beta Q), (E_Q, E_{\beta Q})) \mapsto r({\rm ind_{APS}}(D_{(Q, E_Q)}(\mathcal{P}_{\beta Q}))$$
where
\begin{enumerate}
\item $\mathcal{P}_{\beta Q}$ denotes the pullback from $\beta Q$ to $\beta Q \times P \cong \partial Q$ of a spectral section for the manifold $\beta Q$ and bundle $E_{\beta Q}$; 
\item $D_{(Q, E_Q)}(\mathcal{P}_{\beta Q})$ denotes the Dirac operator on $Q$ twisted by $E_Q$ with the higher APS boundary condition associated with the spectral section $\mathcal{P}_{\beta Q}$;
\item ${\rm ind_{APS}}(\: \cdot \:)$ denotes the higher APS index;
\item $r: K_*(A) \rightarrow K_*(A;\zkz)$ denotes the map from the exact sequence above.
\end{enumerate}
We must show that this map is well-defined; it is not even clear that the map is well-defined at the level of cycles. The analytic data required to define the higher APS index is a metric, compatible connections on both the spinor and vector bundles, and a choice of a spectral section. To show that the above map is independent of these choices, we suppose that we have continuous families of this data. That is, we let
\begin{enumerate}
\item $\{g_t\}_{t\in [0,1]}$ be a one parameter family of Riemannian metrics on $(Q, \beta Q)$;
\item $\nabla_{E_{\beta Q},t}$ be a one parameter family of connections on $E_{\beta Q}$ which is compatible with $g_t|_{\beta Q}$;
\item $\nabla_{E_Q,t}$ be a one parameter family of connections on $E_Q$ which is compatible with $g_t$ and with the family of connections $\nabla_{E_{\beta Q},t}$;
\item $\hat{\mathcal{P}}_0$ and $\hat{\mathcal{P}}_1$ are the pullback of two choices of spectral section for $D_{\beta Q, E_{\beta Q}}$.
\end{enumerate}
Let $D_{(Q, E_Q, 0)}(\hat{\mathcal{P}}_0)$ and $D_{(Q, E_Q, 1)}(\hat{\mathcal{P}}_1)$ the Dirac operators associated to the end points of the family of data fixed above. Also, let $\tilde{D}_t$ denote the family of boundary operators associated to the above data restricted to the boundary; the reader is directed to \cite{LP} for more details on this family of operators. 
\par
With all of the above data fixed, we can apply \cite[Proposition 8 and Theorems 6 and 7]{LP} to obtain
$${\rm ind_{APS}}(D_{(Q, E_Q, 0)}(\hat{\mathcal{P}}_0))-{\rm ind_{APS}}(D_{(Q, E_Q, 1)}(\hat{\mathcal{P}}_1))= {\rm sf}(\tilde{D}_t, \hat{\mathcal{P}}_1, \hat{\mathcal{P}}_0) $$
where ${\rm sf}(\tilde{D}_t, \hat{\mathcal{P}}_1, \hat{\mathcal{P}}_0)$ is the spectral flow of the family of operators, $\tilde{D}_t$ (see \cite{LP} for details). 
\par
Applying the map $r$, observing that right hand side of the equation (i.e., the spectral flow term) is in the image of $k$, and exactness imply the required result: 
$$r({\rm ind_{APS}}(D_{(Q, E_Q, 0)}(\hat{\mathcal{P}}_0)))=r({\rm ind_{APS}}(D_{(Q, E_Q, 1)}(\hat{\mathcal{P}}_1))).$$
\par
Next, we must show that the map respects the three relations used to define $K_*(pt;\{k-\hbox{points}\};A)$. For the direct sum/disjoint union relation the proof follows from basic properties of the higher APS-index. 
\par
For the bordism relation, suppose $((Q, \beta Q), (E_Q, E_{\beta Q}), f)$ is the boundary of the bordism $((\bar{Q}, \beta \bar{Q}), (\bar{E}_{\bar{Q}}, \bar{E}_{\beta \bar{Q}}), g)$. After fixing the required analytic data to define the higher APS index, we have that \cite[Theorems 8 and 9]{LP} imply that
$${\rm ind}_{AS}(D_{\partial \bar{Q}, \bar{E}_{\bar{Q}}|_{\partial \bar{Q}}}) = {\rm ind_{APS}}(D_{(Q, E_Q)}(\mathcal{P})) + k \cdot {\rm ind_{APS}}(D_{(\beta \bar{Q}, \beta \bar{E}_{\bar{Q}})}(\tilde{\mathcal{P}}))  $$
where 
\begin{enumerate}
\item $\pi: \partial Q \cong \beta Q \times P \rightarrow \beta Q$ is the projection map;
\item $\pi^*(\tilde{\mathcal{P}})=\mathcal{P}$; recall that, by assumption, $\mathcal{P}$ is the pullback of a spectral section associated to $\beta Q$ and $\beta E$.
\end{enumerate}
Applying the map $r$ to this equation and using exactness, we obtain
$$ r({\rm ind}_{AS}(D_{\partial \bar{Q}, \bar{E}_{\bar{Q}}|_{\partial \bar{Q}}})) = r({\rm ind_{APS}}(D_{(Q, E_Q)}(\mathcal{P}))). $$
Finally, the bordism invariance of the higher AS-index implies that ${\rm ind}_{AS}(D_{\partial \bar{Q}, \bar{E}_{\bar{Q}}|_{\partial \bar{Q}}})$ vanishes. Hence the proof of the required invariance under the bordism relation is complete. Finally, for vector bundle modification the result follows from \cite[Proposition 4.2]{DeeAnaTopIndMaps}.
\end{ex}
The discussion in the previous example gives a proof of the next theorem. We use the notation introduced in that example and note that the assumption that ${\rm ind_{AS}}(D_{\beta Q, E_{\beta Q}})=0$ implies that there exists a spectral section with respect to $\partial Q$ and $E_Q|_{\partial Q}$ that is the pullback of a spectral section with respect to $\beta Q$ and $E_{\beta Q}$ (see \cite[Section 2]{LP}).
\begin{theorem} \label{higFMThm}
Let $((Q, \beta Q), (E_Q, E_{\beta Q}))$ be a cycle in $K_*(pt;k$-points$;A)$ such that ${\rm ind_{AS}}(D_{\beta Q, E_{\beta Q}})=0$. Then 
$$r({\rm ind_{APS}}(D_{Q, E_Q}(\mathcal{P})))$$
is a topological invariant (i.e., independent of the metric, connection, etc). Moreover, it is also,  in the sense of Definition \ref{borForKthGrp}, a cobordism invariant. 
\end{theorem}
\begin{remark}
In fact, there is a topological formula for $r({\rm ind_{APS}}(D_{Q, E_Q}(\mathcal{P}))$. It is obtained in the same way as in the Freed-Melrose index theorem; that is, one considers an embedding of the $\zkz$-manifold into a suitable half-space (see \cite{Fre} or \cite[Section 1.4]{Dee1}) and defines a wrong-way map in $K$-theory. However, we will not discuss this in detail.
\end{remark}
\begin{ex}
In this example, we work with the bundle model of $K_*(X;P;A)$ in the case when $P$ is the circle, $S^1$. We will show that 
$$K_0(pt; S^1; A) \cong K_0(A) \oplus K_0(A) \hbox{ and } K_1(pt;S^1;A) \cong K_1(A)\oplus K_1(A).$$ 
In particular, if $A = \field{C}$, then $K_0(pt;S^1) \cong \field{Z} \oplus \field{Z}$ and $K_1(pt;S^1)\cong \{0\}$. 
\par
This result can be obtained indirectly using the exact sequence in Theorem \ref{BockTypeSeq}. However, there is an explicit isomorphism defined as follows:
$${\rm ind}_{S^1}: ((Q,\beta Q), (E_{Q}, E_{\beta Q})) \mapsto \left( {\rm ind_{AS}}(D_{(\beta Q, \beta E)}), {\rm ind_{AS}}(D_{Q\cup_{\partial Q} \beta Q \times \field{D}, E_{Q} \cup \pi^*(E_{\beta Q})}) \right)$$
where
\begin{enumerate}
\item $\field{D}$ denotes the unit disk;
\item $E_{Q} \cup \pi^*(E_{\beta Q})$ denotes the vector bundle obtained by clutching via the isomorphism assoicated to the $P$-bundle $(E_Q, E_{\beta Q})$;
\item $D_{M,E}$ denotes the Dirac operator of the manifold $M$ twisted by the bundle $E$;
\item $\pi: \beta Q \times P \rightarrow \beta Q$ denotes the projection map.
\end{enumerate}
We must prove that this map is well-defined. That the map respects the direct sum/disjoint union relation follows using definitions and basic properties of the index map. For the bordism relation, suppose that $((Q,\beta Q), (E_{Q}, E_{\beta Q}))$ is the boundary of the bordism $((\bar{Q},\beta \bar{Q}), (\bar{E}_{Q}, \bar{E}_{\beta Q}))$. We are required to show that 
$${\rm ind}(D_{(\beta Q, \beta E)})=0 \hbox{ and } {\rm ind}(D_{Q\cup_{\partial Q} \beta Q \times\field{D}, E_{Q} \cup \pi^*(E_{\beta Q})})=0.$$
The first of the two holds by the cobordism invariance of the index and the fact that $(\beta Q, \beta E)= \partial (\beta \bar{Q}, \bar{E}_{\beta Q})$.  \par
The proof of the second equality is more involved. Let $\tilde{\pi}: \beta \bar{Q} \times S^1 \rightarrow \beta \bar{Q}$ be the projection map. The definition of bordism as in Definition \ref{borForBunGrp} implies that the cycle $(Q\cup \beta \bar{Q}\times S^1,E\cup\tilde{\pi}^*(\bar{E}_{\beta \bar{Q}}))$ is a boundary; it is the boundary of $(\bar{Q}, \bar{E})$. Moreover, standard results in bordism theory imply that
$$(Q\cup_{\partial Q} \beta \bar{Q}\times S^1, E\cup \tilde{\pi}^*(E_{\beta \bar{Q}})) \dot{\cup} -(Q \cup_{\partial Q} \beta Q \times \field{D}, E_{Q} \cup \pi^*(E_{\beta Q}))$$
is bordant to $$(\beta Q \times \field{D} \cup_{\partial Q} \beta \bar{Q}\times S^1, \pi^*(E_{\beta Q}) \cup \tilde{\pi}^*(E_{\beta \bar{Q}}) ).$$
By ``straightening the angle" (see \cite{CFPerMap} or \cite[Appendix]{Rav}) of $\beta \bar{Q} \times \field{D}$, one can show that $ (\beta Q \times \field{D} \cup_{\partial Q} \beta \bar{Q}\times S^1, \pi^*(E_{\beta Q}) \cup \tilde{\pi}^*(E_{\beta \bar{Q}}))$ is a boundary. The combination of the three bordisms discussed in this paragraph imply that $(Q \cup_{\partial Q} \beta Q \times \field{D}, E_{Q} \cup \pi^*(E_{\beta Q}))$ is a boundary. The cobordism invariance of the index, then implies the required vanishing result.
\par
For the vector bundle modification relation, if $(V, V_{\beta Q})$ is an even rank ${\rm spin^c}$ $P$-vector bundle over $(Q, \beta Q)$, then one can form the vector bundle $V \cup \pi^*(V_{\beta Q})$; it is an even rank ${\rm spin^c}$ vector bundle over $Q \cup \beta Q \times \field{D}$. This fact and the invariance of the index under vector bundle modification in the standard Baum-Douglas model give the required result.
\par
Next, we prove that the following diagram is commutative:
{\footnotesize \minCDarrowwidth10pt\begin{center}
$\begin{CD}
@>>> K_1(pt;A) @>\Phi_{S^1}>> K_0(pt;A) @>r_{S^1}>> K_0(pt;S^1;A) @>\delta_{S^1}>> K_0(pt;A) @>\Phi_{S^1}>> K_1(pt;A) @>>> \\
@. @V{\rm ind} VV @V{\rm ind} VV @V{\rm ind}_{S^1} VV @V{\rm ind} VV  @V{\rm ind}VV @. \\
@>>> K_1(A) @>\Phi >> K_0(A) @>r >> K_0(A) \oplus K_0(A) @>\delta >> K_0(A) @>\Phi>> K_1(A) @>>> 
\end{CD}$
\end{center}}
\noindent
To begin, we note that the $\Phi_{S_1}$ is the zero map because $S^1$ is a boundary. Secondly, if $(M,E)\in K_*(pt;A)$, then 
$$({\rm ind}_{S^1}\circ r_{S^1})(M,E) = ({\rm ind}_{AS}(D_E), 0)= r({\rm ind}(M,E))$$
as required. Finally, if $((Q,\beta Q), (E_Q, E_{\beta Q})) \in K_*(pt;S^1;A)$, then 
$$({\rm ind}\circ \delta_{S^1})((Q,\beta Q), (E_Q, E_{\beta Q}))={\rm ind}_{AS}(D_{(\beta Q, E_{\beta Q})})=(\delta \circ {\rm ind}_{S^1})((Q, \beta Q), (E_Q, E_{\beta Q}))$$
The Five Lemma and the previous commutative diagram imply the main result of this example (i.e., that the index map $K_*(pt;S^1;A)$ to $K_*(A)\oplus K_*(A)$ is an isomorphism). 
\end{ex}
For other examples of $P$, one can also compute $K_*(pt;P;A)$. For example, if $P$ is the two-sphere, $S^2$, then an isomorphism, $K_*(pt;S^2;A) \rightarrow K_{*+1}(A)\oplus K_*(A)$, can defined via
$$((Q, \beta Q), (E_{Q}, E_{\beta Q})) \mapsto  \left( {\rm ind}(D_{(\beta Q, \beta E)}), {\rm ind}(D_{Q\cup_{\partial Q} \beta Q \times \field{D}^2, E_{Q} \cup \pi^*(E_{\beta Q})}) \right)$$
where
\begin{enumerate}
\item $\field{D}^2$ is the two-dimensional ball;
\item $\pi:\beta Q \times \field{D}^2 \rightarrow \beta Q$ is the projection map;
\item $D_{\beta Q, E_{\beta Q}}$ and ${\rm ind_{AS}}$ are defined as in the previous example.
\end{enumerate}
The proof that this map is an isomorphism is similar to the proof that ${\rm ind}_{S^1}$ is an isomorphism given in the previous example. Furthermore, these methods can be used to show that $K_*(pt;S^n;A) \cong K_*(A) \oplus K_{*-n-1}(A)$ for any $n$.
\section{Generalizations}
The constructions in the previous sections of the paper have a number of generalizations.  The most obvious to the reader familiar with the Baum-Douglas model for K-homology are to various equivariant settings. For the corresponding equivarent theories one should replace our cycles (i.e., Definition \ref{cycWithKth}) with the natural generalization of the cycles in \cite{BOOSW} in the case of a compact Lie group or \cite{BHSProGrpAct} in the case of a discrete group which acts properly. Both these two generalization can be obtained using standard techniques and tools developed to this point.  For example, in the case of proper discrete group actions, the reader should see \cite[Section 4.5]{Rav} for more on the correct notion of ``normal bordism". We will not give a detailed developement. On the other hand, in the case of the action of a groupoid the correct generalization is less clear, but would be an interesting project. The reader might find \cite{EM} a useful starting point. This completes the discussion of the equivariant generalizations.
\par
Another generalization is as follows. Based on the statement of Theorem \ref{uniquenessThm}, one is lead to consider the map on K-homology induced by more general cycles in $K_*(pt)$; the precise definition of the relevant cycles is given below in Definition \ref{defGenCyc}. Let $(P,F)$ be a cycle in $K_*(pt)$, then we can consider the map $\Phi_{(P,F)}: K_*(X;A) \rightarrow K_{*+{\rm dim}(P)}(X;A)$ defined at the level of cycles via
$$(M,E,f) \mapsto (M \times P, (\pi_M)^*(E) \otimes (\pi_P)^*(F), f \circ \pi_M)$$
where $\pi_M$ (resp. $\pi_P$) denote the projection map from $M\times P$ to $M$ (resp. $P$). One is naturally led to ask if the constructions in the previous sections generalize to this setting; this is the case so long as we continue to assume $P$ has trivial stable normal bundle. To be precise, we have the following definition of a cycle with respect to $(P,F)$:
\begin{define} \label{defGenCyc}
A geometric cycle with bundle data, over $X$ with respect to $(P,F)$ and $A$ is a triple $((Q,\beta Q), (E_Q, E_{\beta Q}), f)$ where
\begin{enumerate}
\item $(Q,\beta Q)$ is a compact, smooth, ${\rm spin^c}$ $P$-manifold;
\item $E_Q$ is a smooth $A$-vector bundle over $Q$;
\item $E_{\beta Q}$ is a smooth $A$-vector bundle over $\beta Q$;
\item $\alpha: E_Q|_{\partial Q} \rightarrow (\pi_{\beta Q})^*(E_{\beta Q}) \otimes (\pi_{P})^*(F)$ is an isomorphism of $A$-vector bundles; we note that $\pi_{\beta Q}$ (respectively $\pi_P$) denotes the projection map from $\partial Q \cong \beta Q\times P$ to $\beta Q$ (respectively $P$);
\item $f$ is a continuous map from $(Q,\beta Q)$ to $X$.
\end{enumerate}
\end{define}
The reader should note that when $F=P\times \field{C}$, Definition \ref{defGenCyc} is exactly the same as Definition \ref{cycWithBun}. Let $K_*(X;(P,F);A)$ denote the set of isomorphism classes of cycles modulo the natural generalization of the equivalence discussed in Definition \ref{GeoGroup}. The generalizations of the two main theorems in Section \ref{geoModSec} are as follows. We will not give the proofs since they can be obtained using our previous proofs with only additional notational complexity.
\begin{theorem}
Let $X$ be a finite CW-complex, $A$ be a unital $C^*$-algebra, $P$ be a smooth compact $spin^c$-manifold that has trivial stable normal bundle and well-defined dimension modulo two, and $F$ be a vector bundle over $P$. Then, the following sequence is exact
\begin{center}
$\begin{CD}
K_0(X;A) @>\Phi>> K_{{\rm dim}(P)}(X;A) @>r>> K_{{\rm dim}(P)}(X;(P,F);A) \\
@A\delta AA @. @V\delta VV \\
 K_{{\rm dim}(P)+1}(X;(P,F);A) @<r<<  K_{{\rm dim}(P)+1}(X;A) @<\Phi<<  K_1(X;A)
\end{CD}$
\end{center}
where the maps are defined as follows
\begin{enumerate}
\item $\Phi$ is defined at the level of cycles via 
$$(M,E,f) \mapsto (M \times P, (\pi_M)^*(E) \otimes (\pi_P)^*(F), f \circ \pi_M)$$ 
where $\pi_M$ (respectively $\pi_P$) denotes the projection map from $M\times P$ to $M$ (respectively $P$);
\item $r$ is defined at the level of cycles via 
$$ (M, E, f) \mapsto ((M,\emptyset), E, f);$$ 
\item $\delta$ is defined at the level of cycles via 
$$((Q,\beta Q), (E_Q,  E_{\beta Q}), f) \mapsto (\beta Q,  E_{\beta Q}, f|_{\beta Q}).$$
\end{enumerate}
\end{theorem}
\begin{theorem}
Let $X$ be a finite CW-complex, $A$ be a unital $C^*$-algebra, $P$ and $P^{\prime}$ be smooth, compact, ${\rm spin^c}$-manifolds, each of which has a trivial stable normal bundle and well-defined dimension, and $F$ and $F^{\prime}$ be vector bundles over $P$ and $P^{\prime}$ respectively. If $(P, [F])\sim (P^{\prime}, [F^{\prime}])$ as elements in $K_*(pt)$, then $K_*(X;(P,F);A) \cong K_*(X;(P^{\prime},F^{\prime});A)$; moreover, the isomorphism is natural (with respect to both $X$ and $A$).
\end{theorem}
\begin{ex}
Let $(P,F)=(S^2, F_{{\rm Bott}})$ where $F_{{\rm Bott}}$ is the Bott bundle (see for example \cite{BD}). Then the map $\Phi_{(S^2,F_{{\rm Bott}})}$ is the Bott periodicity isomorphism; hence $K_*(X;(P,F);A)\cong \{0\}$.
\end{ex} 
%\vspace{0.2cm} 
\noindent
\\
{\bf Acknowledgments} \\
The author thanks Magnus Goffeng and Thomas Schick for discussions.  This work began while the auther held an NSERC Postdoctoral Fellowship at Georg-August Universit${\rm \ddot{a}}$t, G${\rm \ddot{o}}$ttingen.  

\vspace{0.25cm}
Email address: robin.deeley@gmail.com \vspace{0.25cm} \\
{ \footnotesize Universit${\rm \acute{e}}$ Blaise Pascal, Clermont-Ferrand II, Laboratoire de Math${\rm \acute{e}}$matiques, Campus des C${\rm \acute{e}}$zeaux B.P. 80026 63177 Aubi${\rm \grave{e}}$re cedex, France}
\end{document}